\documentclass{amsart}
\usepackage{amsmath}%
\usepackage{amsthm}
\usepackage{amsfonts}%
\usepackage{amssymb}%
\usepackage{graphicx}
\usepackage{bm}
\usepackage{comment}
\usepackage{mathtools}
\usepackage{bbm}

\usepackage[all,cmtip]{xy}
\usepackage{tikz-cd}
\usetikzlibrary{cd}

\usepackage[utf8]{inputenc}
\usepackage{a4}
\usepackage[hidelinks]{hyperref}
\usepackage{enumitem}

\usepackage{mathrsfs}

\DeclareFontFamily{OT1}{pzc}{}
\DeclareFontShape{OT1}{pzc}{m}{it}{<-> s * [1.0] pzcmi7t}{}
\DeclareMathAlphabet{\mathpzc}{OT1}{pzc}{m}{it}

\usepackage[singlespacing]{setspace}
\newtheorem{theorem}{Theorem}[section]

\newtheorem{corollary}[theorem]{Corollary}

\newtheorem{definition}[theorem]{Definition}

\newtheorem{lemma}[theorem]{Lemma}

\newtheorem{proposition}[theorem]{Proposition}

\numberwithin{equation}{section}
\def\XXint#1#2#3{{\setbox0=\hbox{$#1{#2#3}{\int}$}
\vcenter{\hbox{$#2#3$}}\kern-.5\wd0}}

\newcommand{\R}{\mathbb{R}}

\newcommand{\N}{\mathbb{N}}
\newcommand{\Z}{\mathbb{Z}}

\newcommand{\Ha}{\mathcal{H}}
\newcommand{\leb}{\mathcal{L}}

\newcommand{\spt}{\operatorname{spt}}

\newcommand{\diam}{\operatorname{diam}}

\newcommand{\Lip}{\operatorname{Lip}}

\newcommand{\on}{\:\mbox{\rule{0.1ex}{1.2ex}\rule{1.1ex}{0.1ex}}\:}

\DeclareMathOperator{\bI}{\mathbf{I}}

\DeclareMathOperator{\bM}{\mathbf{M}}
\DeclareMathOperator{\len}{\textup{length}}
\DeclareMathOperator{\bN}{\mathbf{N}}
\DeclareMathOperator{\trace}{\textup{tr}}

\DeclareMathOperator{\area}{\textup{Area}}

\DeclareMathOperator{\apmd}{\operatorname{apmd}}
\DeclareMathOperator{\mass}{\mathbf{M}}

\newcommand{\bb}[1]{\llbracket #1\rrbracket}
\usepackage{stmaryrd}
\newcommand{\norm}[1]{\lVert#1\rVert}

\usepackage[colorinlistoftodos]{todonotes}
\usepackage{import}
\usepackage{xifthen}
\usepackage{pdfpages}
\usepackage{transparent}
\usepackage{xcolor}
\usepackage{stmaryrd}

\makeatletter
\newcommand*{\cone}{%
	{%
		\mathpalette\@coneOf{\times}%
	}%
}
\newcommand*{\@coneOf}[2]{%
	\sbox0{$\m@th#1\mathsf{#2}$}%
	\mathsf{#2}%
	\kern-\wd0 %
	\mkern2.00mu\relax
	\nonscript\mkern0.25mu\relax
	\mathsf{#2}%
}

\usepackage{etoolbox}

\makeatletter
\patchcmd{\@setaddresses}{\indent}{\noindent}{}{}
\patchcmd{\@setaddresses}{\indent}{\noindent}{}{}
\patchcmd{\@setaddresses}{\indent}{\noindent}{}{}
\patchcmd{\@setaddresses}{\indent}{\noindent}{}{}
\makeatother

\thanks{D.M. was supported by Swiss National Science Foundation grant 212867.}

\author{Denis Marti}

\address{Department of Mathematics\\ University of Fribourg\\  Chemin du Mus\'ee 23\\  1700 Fribourg, Switzerland}
\email{denis.marti@unifr.ch}

\title[One-dimensional and codimension one homology of metric manifolds]{One-dimensional and codimension one homology of metric manifolds}

\begin{document}

\begin{abstract}
We compare singular homology and homology via integral currents in metric spaces that are homeomorphic to smooth manifolds. For such spaces, we provide sufficient conditions that guarantee the existence of a surjective homomorphism from the codimension one homology group via integral currents to the codimension one singular homology group. Moreover, we show that a one-dimensional isoperimetric inequality for integral currents implies that the one-dimensional homology groups coincide. 
\end{abstract}

\maketitle

% \todo[inline]{add Acknowledgements!!!!!!!!!!!!!!!!!!!!!!!!}

\section{Introduction}
    In their seminal paper \cite{fed-flem}, Federer and Fleming introduced integral currents in Euclidean space and the associated homology groups. Ambrosio and Kirchheim later extended this theory to metric spaces \cite{ambrosio-kirchheim-2000}, making it possible to consider homology groups via integral currents in a more general context. These homology groups have become an important topic in metric geometry; see e.g. \cite{depauw-pfeffer-hardt,quasiflats,song2022spherical,wenger-flat-conv}. They are a useful tool for studying the analytic and geometric structure of metric spaces. For instance, in a closed Lipschitz manifold, the homology groups via integral currents and the singular homology groups coincide. This provides a measure-theoretic and analytic representation of singular homology classes. The aim of this article is to explore the relationship between singular homology and homology via integral currents in more general spaces. More specifically, we consider metric spaces that are homeomorphic to a smooth manifold. Such spaces, known as metric manifolds, have a strong connection between their topological and analytic structures, making them suitable for our purposes. Indeed, in joint work with Basso and Wenger \cite{basso2023geometric}, the author proved that, under mild assumptions, the top-dimensional homology group via integral currents in a metric manifold is infinite cyclic. The generator, called the metric fundamental class, was further studied in \cite{denis-non-orientable,marti2024characterization}.

    \begin{definition}\label{def: metric fundamental class}
        Let $X$ be a metric space with finite Hausdorff $n$-measure that is homeomorphic to a closed, oriented, connected Riemannian $n$-manifold $M$. A metric fundamental class of $X$ is an integral current $T\in \bI_n(X)$ with $\partial T=0$ such that
        \begin{itemize}
            \item[(a)] there exists $C>0$ such that $\norm{T}\leq C \cdot \Ha^n$;
            \item[(b)] $\varphi_\# T = \deg(\varphi) \cdot \bb{M}$ for every Lipschitz map $\varphi \colon X \to M$.
        \end{itemize}
    \end{definition}

    Here, $\bb{M}$ denotes the fundamental class of $M$ represented by an integral current, $\deg(\varphi)$ is the (topological) degree of $\varphi$, and $\norm{T}$ is the mass measure of $T$. It follows from \cite[Proposition 5.5]{basso2023geometric} that if a metric manifold $X$ has a metric fundamental class $T$, then $T$ generates the top-dimensional homology group via integral currents $H_{n}^{IC}(X)$. In particular, $H_{n}^{IC}(X)$ is isomorphic to the top-dimensional singular homology group. We show that the existence of a metric fundamental class also has implications for the codimension one homology group via integral currents.

    \begin{theorem}\label{thme: codimension-one-homology}
        Let $X$ be a metric space with finite Hausdorff $n$-measure that is homeomorphic to a closed, oriented, connected smooth $n$-manifold. Suppose that $X$ has a metric fundamental class. Then, there exists a surjective homomorphism from $H_{n-1}^\textup{IC}(X)$ to $H_{n-1}(X)$.
    \end{theorem}
    
    Here, $H_{n-1}^{IC}(X)$ and $H_{n-1}(X)$ denote the \((n-1)\)th homology group via integral currents and the \((n-1)\)th singular homology group of $X$, respectively. We refer to Section \ref{sec: homology} for the precise definitions. The following type of metric manifolds satisfies the conditions of Theorem \ref{thme: codimension-one-homology}. By \cite[Theorem 1.3]{basso2023geometric}, every metric surface ($n=2$) has a metric fundamental class. In addition, any metric manifold that is the Gromov–Hausdorff limit of Lipschitz manifolds with bounded volume has a metric fundamental class \cite[Theorem 1.2]{marti2024characterization}. Finally, \cite[Theorem 1.1]{basso2023geometric} implies that linearly locally contractible metric manifolds support a metric fundamental class. 
    %The following types of metric manifolds have a metric fundamental class: metric surfaces ($n=2)$ by \cite[Theorem 1.3]{basso2023geometric}, Gromov-Hausdorff limits of Lipschitz manifolds with bounded volume \cite[Theorem 1.2]{marti2024characterization} and metric manifolds that are linearly locally contractible \cite[Theorem 1.1]{basso2023geometric}
    A metric space is said to be linearly locally contractible if there exists $\lambda>0$ such that every ball of radius  $0<r<\diam( X)/ \lambda$ is contractible within the ball with the same center and radius $\lambda r$. We note that, in general, the homomorphism in Theorem \ref{thme: codimension-one-homology} is not an isomorphism. Indeed, in Section \ref{sec: examples}, we construct two examples of geodesic metric spaces with finite Hausdorff $2$-measure that are homeomorphic to the two-sphere, but their $1$-dimensional homology groups via integral currents are not trivial. The first example is Ahlfors $2$-regular but not linearly locally contractible. On the other hand, the second example is linearly locally contractible but not Ahlfors $2$-regular. We call a metric space $X$ Ahlfors $2$-regular if there exists a constant $c>0$ such that for every ball $B(x,r)$ in $X$, we have
    $$c^{-1}r^2 \leq \Ha^2(B(x,r))\leq cr^2.$$ 
    Both examples in Section \ref{sec: examples} have a metric fundamental class. Therefore, for a large group of metric surfaces, the one-dimensional homology group via integral currents is strictly larger than the one-dimensional singular homology group. However, in certain situations, they coincide.

    \begin{theorem}\label{thme: one-homology}
        Let $X$ be a quasiconvex metric space that is homeomorphic to a closed, connected smooth manifold. If $X$ satisfies a $1$-dimensional small mass isoperimetric inequality for integral currents, then $H_1^\textup{IC}(X)$ and $H_1(X)$ are isomorphic.
    \end{theorem}

    A metric space $X$ is said to be quasiconvex if there exists a constant $c>0$ such that for all $x,y\in X$, there exists a curve $\gamma$ from $x$ to $y$ with $\len(\gamma) \leq cd(x,y)$. The definition of the isoperimetric inequality for integral currents is presented in Section \ref{see: isoper}. There, we also introduce the local quadratic isoperimetric inequality for Lipschitz curves. This type of isoperimetric inequality plays an important role in analysis in metric spaces; see e.g. \cite{Lytchack-wenger-Canonical-parameterizations, Lytchak-wenger-robert, meier2025monotone}. It is well-known that a metric surface that is Ahlfors $2$-regular and linearly locally contractible satisfies a local quadratic isoperimetric inequality for Lipschitz curves. Since a local quadratic isoperimetric inequality for Lipschitz curves implies a $1$-dimensional small mass isoperimetric inequality for integral currents (Proposition \ref{prop: sobolev-isop-implies-current-isoper}), we obtain the following corollary.
    
    %We give the definition of the isoperimetric inequality for integral currents in Section \ref{see: isoper}. There, we also introduce the local quadratic isoperimetric inequality for Lipschitz curves. This type of isoperimetric inequality plays an important role in analysis in metric spaces; see e.g. \cite{Lytchack-wenger-Canonical-parameterizations, Lytchak-wenger-robert, meier2025monotone}. It is well-known that a metric surface that is Ahlfors $2$-regular and linearly locally contractible satisfies a local quadratic isoperimetric inequality for Lipschitz curves. Since a local quadratic isoperimetric inequality for Lipschitz curves implies a $1$-dimensional small mass isoperimetric inequality for integral currents (Proposition \ref{prop: sobolev-isop-implies-current-isoper}), we obtain the following corollary.

    \begin{corollary}\label{cor: one-dim-homology-surf}
        Let $X$ be a metric space homeomorphic to a closed, oriented, connected smooth surface. Suppose that $X$ is Ahlfors $2$-regular and linearly locally contractible. Then, $H_k^\textup{IC}(X)$ and $H_k(X)$ are isomorphic for every $k \geq 0$.
    \end{corollary}

    The examples in Section \ref{sec: examples} show that all the assumptions in the corollary are necessary to obtain an isomorphism between $H_1^\textup{IC}(X)$ and $H_1(X)$.

    \medskip

    The article is structured as follows. In Section \ref{sec:prelis}, we fix the notation and present the theory of metric currents, homology theory, and isoperimetric inequalities. We also define the slice operator $T \mapsto \langle T , g, p\rangle$ for Lipschitz maps $g\colon X \to \mathbb{S}^1$ in this section. In the next section, we prove Theorem \ref{thme: codimension-one-homology}. The proof is based on the following classical fact. Let $M$ be a closed, oriented, connected Riemannian $n$-manifold. Given a primitive homology class $[\eta] \in H_{n-1}(M)$, there exists a smooth map $f \colon M \to \mathbb{S}^1$ such that for each regular value $p \in \mathbb{S}^1$, the preimage $N_p = f^{-1}(p)$ is a smooth submanifold that represents $[\eta]$. Using that $H_{n-1}^\textup{IC}(M)$ and $H_{n-1}(M)$ are isomorphic, we can prove the analogous statement for elements in $H_{n-1}^\textup{IC}(M)$. Now, let $X$ be a metric space homeomorphic to $M$ and suppose that $X$ has a metric fundamental class $T \in \bI_n(X)$. We approximate the homeomorphism $\varrho\colon X \to M$ by a Lipschitz map $\varphi \colon X \to M$ and consider the slices $\langle T, f\circ \varphi, p\rangle$ and $\langle\bb{M}, f,p\rangle$ for $p \in \mathbb{S}^1$. Recall that $\bb{M}$ denotes the fundamental class of $M$ represented by an integral current. It follows that for almost all $p \in \mathbb{S}^1$, the slice $\langle\bb{M}, f,p\rangle$ is equal to the integral current $\bb{N_p}$ induced by the smooth submanifold $N_p$. Since the pushforward and the slice operator commute and $T$ is a metric fundamental class of $X$, we have
    $$\varphi_\#\langle T,f \circ \varphi, p\rangle = \langle\varphi_\#T,f,p\rangle = \langle \bb{M},f,p\rangle = \bb{N_p}$$
    for almost all $p \in \mathbb{S}^1$. Therefore, for almost all $p\in \mathbb{S}^1$, the slice $\langle T,f \circ\varphi , p\rangle$ induces a homology class $[\xi]\in H_{n-1}^\textup{IC}(X)$ that is sent to $[\eta]$ by the homomorphism $H_{n-1}^\textup{IC}(X)\to H_{n-1}^\textup{IC}(M)$ induced by $\varphi$. We repeat this for each generator of $H_{n-1}^\textup{IC}(M)$ and obtain a surjective homomorphism $H_{n-1}^\textup{IC}(X)\to H_{n-1}^\textup{IC}(M)$. Finally, we note that $H_{n-1}^\textup{IC}(M)$ and $ H_{n-1}(M)$ as well as $H_{n-1}(M)$ and $H_{n-1}(X)$ are isomorphic, respectively. This concludes the proof of Theorem \ref{thme: codimension-one-homology}.   
    The proof of Theorem \ref{thme: one-homology} is contained in Section \ref{sec: one-dime}. First, we show that any homology class $[\eta] \in H_1^\textup{IC}(X)$ can be represented by an integral $1$-current induced by finitely many Lipschitz loops. This follows from a decomposition result for integral $1$-currents and the $1$-dimensional isoperimetric inequality. A finite collection of Lipschitz loops induces a class in $H_1(X)$. We define a homomorphism $\Theta \colon H_1^\textup{IC}(X) \to H_1(X)$ by sending a class $[\eta]\in H_1^\textup{IC}(X)$ to the singular homology class induced by a representation of $[\eta]$ by finitely many Lipschitz loops. It is not difficult to prove that two different representations of $[\eta]$ induce the same homology class in $H_1(X)$. Since $X$ is geodesic, a straightforward argument shows that $\Theta$ is surjective. Injectivity is more complicated. Let $[\eta] \in H_1^\textup{IC}(X)$ be such that $\Theta([\eta]) = 0$. Then, there exists a representative $T\in \bI_1(X) $ of $[\eta]$ given by finitely many Lipschitz loops and a singular $2$-chain $c$ that bounds $T$. We use $c$ and the isoperimetric inequality to construct an integral $2$-current $V \in \bI_2(X)$ that satisfies $\partial V = T$. Thus, $[\eta]  = 0$, which shows that $\Theta$ is injective. Finally, in Section \ref{sec: examples}, we construct the two examples mentioned in the introduction. 

\section{Preliminaries}\label{sec:prelis}
    \subsection{Metric notions}
        Let $(X,d)$ be a metric space. We write $B(x,r)$ for the open ball with center $x \in X$ and radius $r> 0$. For $A \subset X$ and $r>0$, we denote the open $r$-neighborhood of $A$ by 
        $$N_r^X(A) = \left\{x \in X \colon d(A,x) < r\right\}.$$
        If the ambient space $X$ is clear from the context, we simply write $N_r(A)$. A map $f\colon X \to Y$ between metric spaces is said to be \textit{\(L\)-Lipschitz} if $d(f(x),f(y)) \leq L d(x,y)$ for all $x,y \in X$. The smallest $L\geq 0$ satisfying the previous inequality is called the Lipschitz constant of $f$ and is denoted by $\Lip(f)$. If $f$ is injective and the inverse $f^{-1}$ is $L$-Lipschitz as well, we say $f$ is \textit{$L$-bi-Lipschitz}. We let $\Lip(X)$ be the set of all Lipschitz functions from $X$ to $\R$. The Hausdorff $k$-measure on $X$ is denoted by $\Ha^k$. Given two maps $f, g\colon X \to Y$, we define the uniform distance between \(f\) and \(g\) by
        $$ d(f, g)=\sup_{x\in X} d(f(x), g(x)).$$
        We need the following Lipschitz approximation result.
        
        \begin{lemma}\label{lemma: lip-approx-easy}
            Let \(f\colon X \to M\) be a continuous map from a compact metric space \(X\) to a closed Riemannian manifold $M$. Then, for every \(\varepsilon>0\), there exists a Lipschitz map \(g \colon X \to M\) with \(d(f, g)<\varepsilon\). 
        \end{lemma}

        The lemma is a particular instance of \cite[Lemma 2.1]{basso2023geometric} since every closed Riemannian manifold is an absolute Lipschitz neighborhood retract \cite[Theorem~3.1]{hohti-1993}. We will use the previous lemma frequently in the following context. Let $\varrho\colon X \to M$ be a homeomorphism, where $X$ is a metric space and $M$ is a closed Riemannian manifold. By Lemma \ref{lemma: lip-approx-easy}, there exist Lipschitz maps $X \to M$ that are arbitrarily close to $\varrho$. If we choose the Lipschitz map $\varphi\colon X \to M$ sufficiently close to $\varrho$, then $\varphi$ is homotopic to $\varrho$, and the composition $\varrho^{-1}\circ \varphi$ is homotopic to the identity on $X$. Indeed, embed $M$ in $l^\infty$, and let $U \subset l^\infty$ be an open neighborhood of $M$ such that there exists a retraction $\pi \colon U \to M$. Whenever $d(\varphi,\varrho)$ is sufficiently small, the straight-line homotopy $H$ between $\varphi$ and $\varrho$ stays within $U$. Hence, $\pi \circ H$ is well-defined and is a homotopy between $\varphi$ and $\varrho$ as maps from $X$ to $M$. Since $X$ is a topological manifold, it is an absolute neighborhood retract. Thus, an analogous argument shows that $\varrho^{-1}\circ \varphi$ is homotopic to the identity on $X$, in case $\varphi\colon X \to M$ is sufficiently close to $\varrho$.

\begin{comment}       
        \todo[inline]{add comment that close maps in manifolds are homotopic}
        A metric space \(Y\) is an \textit{absolute Lipschitz neighborhood retract} if there exists \(C>0\) such that whenever \(Y\subset Z\) for some metric space \(Z\), then there exists an open neighborhood $U\subset Z$ of \(Y\) and a \(C\)-Lipschitz retraction \(R\colon U \to Y\). Important examples of absolute Lipschitz neighborhood retracts are closed Riemannian manifolds, see e.g. \cite[Theorem~3.1]{hohti-1993}. In particular, using this fact and Lemma \ref{lemma: lip-approx-easy}, it is not difficult to show that the following holds. If $M$ is a closed Riemannian manifold and $\varepsilon>0$, there exists $\delta>0$ such that whenever $f,g\colon X \to M$ are Lipschitz maps from a metric space $X$ into $M$ with $d(f,g) < \delta$, then there exists a Lipschitz homotopy $H:f\simeq g$ satisfying $d(H(x,t),H(x,s)) \leq \varepsilon$ for every $x \in X$ and $t$.
\end{comment}

    \subsection{Metric currents}\label{sec: currents}
        We provide a brief introduction to the theory of currents in metric spaces and record some facts we will need later. For more information, we refer to \cite{ambrosio-kirchheim-2000} and \cite{lang-local}. Let $X$ be a complete metric space and $k \geq 0$ be an integer. Let  $\mathcal{D}^k(X) = \Lip_b(X) \times \Lip(X)^k$, where $\Lip_b(X)$ denotes the space of bounded Lipschitz functions $X \to \R$. A multilinear functional $T$ on $\mathcal{D}^k(X)$ is said to be a metric $k$-current if it satisfies a certain continuity, locality, and finite mass property. For each $k$-current $T$, there exists an associated finite Borel measure $\norm{T}$, called the mass measure of $T$. This measure can be thought of as the volume measure on the generalized surface $T$. We write $\bM_k(X)$ for the space of metric $k$-currents on $X$. Let $T \in \bM_k(X)$. We denote by $\bM(T) = \norm{T}(X)$ the mass of $T$. There exists a unique extension of $T$ to $L^1(X,\norm{T})\times\Lip(X)^k$. For a Borel set $B \subset X$, we define the restriction $T\on B \in \bM_k(X)$ by
        $$T\on B (f,\pi_1,\dots,\pi_k) = T(\mathbbm{1}_B f,\pi_1,\dots,\pi_k)$$
        for all $(f,\pi_1,\dots,\pi_{k})\in \mathcal{D}^{k}(X)$.
        The boundary of $T$ is given by
        $$\partial T (f,\pi_1,\dots,\pi_{k-1}) = T(1,f,\pi_1,\dots,\pi_{k-1})$$
        for all $(f,\pi_1,\dots,\pi_{k-1})\in \mathcal{D}^{k-1}(X)$. If the boundary $\partial T$ is a $(k-1)$-current, then we call $T$ a normal $k$-current. The space of normal $k$-currents in $X$ is denoted by $\bN_k(X)$. We define the normal mass by $\bN(T) = \bM(T) + \bM(\partial T)$. The space $\bN_k(X)$ equipped with the normal mass $\bN$ is a Banach space. Given a Lipschitz map $\varphi\colon X \to Y$, the pushforward of $T$ under $\varphi$ is the metric $k$-current in $Y$ defined as
        $$\varphi_\# T(f,\pi_1,\dots,\pi_k) = T(f\circ \varphi,\pi_1\circ \varphi,\dots,\pi_k\circ\varphi)$$
        for all $(f,\pi_1,\dots, \pi_{k})\in \mathcal{D}^k(Y)$. In Euclidean space, every function $\theta \in L^1(\R^k)$ induces a metric $k$-current as follows
        $$\bb{\theta}(f,\pi_1,\dots,\pi_k) = \int_{\R^k} \theta f \det(D \pi) \; d\leb^k$$
        for all $(f,\pi) = (f,\pi_1,\dots,\pi_k) \in \mathcal{D}^k(\R^k)$.  A normal $k$-current $T \in \bN_k(X)$ is called an integral $k$-current if there exist countably many compact sets $K_i \subset \R^k$ and $\theta_i \in L^1(\R^k,\Z)$ with $\spt \theta_i \subset K_i$ and bi-Lipschitz maps $\varphi_i \colon K_i \to X$ such that
        $$T=\sum_{i\in \N} \varphi_{i\#}\bb{\theta_i} \quad \text{ and } \quad \mass(T)=\sum_{i\in \N} \mass\big(\varphi_{i\#}\bb{\theta_i}\big).$$
        We denote by $\bI_k(X)$ the space of integral $k$-currents in $X$. An integral $k$-current $T$ is called an integral $k$-cycle if it has zero boundary $\partial T = 0$. Important examples of integral cycles are currents induced by Riemannian manifolds. Let $M$ be a closed, oriented Riemannian $k$-manifold. Then 
        \[
        \bb{M}(f,\pi) = \int_Mf\det(D\pi)\,d\hspace{-0.14em}\Ha^k
        \]
        for all \((f, \pi)\in \mathcal{D}^k(M)\), defines an integral $k$-current on $M$. Stokes' theorem implies that $\bb{M}$ is a cycle. In case $M$ is connected, the integral cycle $\bb{M}$ generates $H_k^\textup{IC}(M)$ and is called the fundamental class of $M$.

        \medskip

        Next, we explain the slicing of a normal current. This is an important technique that allows us to represent a current with lower dimensional pieces of the current. Let $T \in \bN_k(X)$ and let $f\colon X \to \R$ be Lipschitz. For almost all $t \in \R$, the slice $\langle T, f, t\rangle$ of $T$ by $f$ at $t$ is a normal $(k-1)$-current characterized by the following property:
        \begin{equation}\label{eq: charact-property-slice}
            \int_\R \langle T, f, t\rangle \psi(t) \; dt = T\on(\psi \circ f)df
        \end{equation}
        for all $\psi \in C_c(\R)$; see \cite[Theorem 5.6]{ambrosio-kirchheim-2000}. Here, $ T\on(\psi \circ f)df$ is the $(k-1)$-current defined by 
        $$T\on(\psi \circ f)df(g,\pi) = T((\psi\circ f) g, f, \pi)$$
        for all $(g,\pi) \in \mathcal{D}^{k-1}(X)$. Using the characterizing property of the slice, we show that slicing commutes with the pushforward by a Lipschitz map.

        \begin{lemma}\label{lemma: pushfw-slice-is-slice-phsfwd-normal}
            Let $\varphi \colon X \to Y$ be a Lipschitz map between two complete metric spaces. Furthermore, let $T \in \bN_k(X)$ and $f\colon Y\to\R$ be Lipschitz. Then 
            $$\varphi_\#\langle T, f\circ \varphi, t \rangle = \langle \varphi_\# T, f, t \rangle$$
            for almost all $t \in \R$.
        \end{lemma}
    
        \begin{proof}
            Set $g = f\circ \varphi$. Fix $(h,\pi) \in \mathcal{D}^{k-1}(Y)$ and $\psi \in C_c(\R)$ for the moment. Then
            $$\varphi_\#(T\on (\psi \circ g)dg)(h,\pi)= T\on (\psi\circ g)dg(h\circ \varphi,\pi\circ\varphi).$$
            Thus, by the characterizing property of the slice \eqref{eq: charact-property-slice}, we have
            \begin{align*}
                \varphi_\#(T\on (\psi \circ g)dg)(h,\pi) &=\int_{\R} \langle T, g, t \rangle(h\circ \varphi,\pi\circ\varphi) \; \psi(t) \; dt
                \\
                &= \int_{\R}\varphi_\#\langle T, g, t \rangle(h,\pi) \; \psi(t) \; dt.
            \end{align*}
            On the other hand,
            $$\varphi_\#(T\on (\psi \circ g)dg)(h,\pi) = T(((\psi\circ f)h)\circ \varphi, f \circ \varphi, \pi \circ \varphi)= (\varphi_\#T)\on(\psi\circ f)df (h,\pi).$$
            Therefore, by applying the characterizing property of the slice to $\varphi_\#T$ this time, we get
            \begin{equation*}\label{eq: slice-pushfwd-2}
                (\varphi_\#(T\on (\psi \circ g)dg))(h,\pi) = \int_{\R} \langle \varphi_\# T, f, t\rangle(h,\pi)\; \psi(t) \; dt.
            \end{equation*}
            We conclude that 
            $$\int_{\R}\varphi_\#\langle T, g, t \rangle(h,\pi) \; \psi(t) \; dt = \int_{\R} \langle \varphi_\# T, f, t\rangle(h,\pi)\; \psi(t) \; dt$$
            for all $(h,\pi) \in \mathcal{D}^{k-1}(Y)$ and every $\psi \in C_c(\R)$. The functions $t \mapsto \varphi_\#\langle T, g, t \rangle(h,\pi)$ and $t \mapsto \langle \varphi_\#T, f, t \rangle(h,\pi)$ are $\leb^1$-integrable for all $(h,\pi) \in \mathcal{D}^{k-1}(Y)$. Hence, the fundamental lemma of calculus of variations implies that for almost all $t \in \R$
            $$\varphi_\#\langle T, g, t \rangle(h,\pi) = \langle \varphi_\#T, f, t \rangle(h,\pi).$$
            This completes the proof.
        \end{proof}

        \medskip
        If $T$ is an integral current and $f \colon X \to \R$ is Lipschitz, then for almost all $t \in \R$, the slice $\langle T,f,t\rangle$ is an integral current as well \cite[Theorem 5.7]{ambrosio-kirchheim-2000}. In this case, we have an explicit representation of the slice. We first introduce some terminology and refer to \cite[Section 9]{ambrosio-kirchheim-2000} and \cite{ambrosio-kirchheim-rect} for more information on the concepts discussed below. Suppose that $X$ is separable and embed $X$ into $l^\infty$. If $E \subset l^\infty$ is $k$-rectifiable, then $E$ has an approximate tangent space $\textup{Tan}(E,x)$ at almost all $x\in E$. The approximate tangent space is a $k$-dimensional linear subspace of $l^\infty$. An orientation $\tau$ of $E$ is a choice of unit simple $k$-vectors $\tau=\tau_1\wedge\dots\wedge\tau_k$ such that the $\tau_i \colon E \to l^\infty$ are Borel functions that span the approximate tangent space $\textup{Tan}(E,x)$ for almost every $x\in E$. If $\pi\colon l^\infty\to \R$ is Lipschitz, then $\pi$ is tangentially differentiable almost everywhere on $E$. The tangential differential of $\pi$ at $x$ is a linear map
        $$d_x^E\pi \colon \textup{Tan}(E,x) \to \R$$
        that converges to the difference quotient of $\pi$ at $x$ in a suitable sense. For $\pi = (\pi_1,\dots,\pi_k) \colon l^\infty\to \R^k$ Lipschitz, we let $\bigwedge_kd_x^E \pi$ be the simple $k$-covector induced by the tangential differentials of the components of $\pi$. Let $T\in \bI_k(X)$. It follows from \cite[Theorem 9.1]{ambrosio-kirchheim-2000} that there exist a $k$-rectifiable set $E \subset X\subset l^\infty$ with finite Hausdorff $k$-measure, a Borel function $\theta\colon E \to \N$, and an orientation $\tau$ of $E$ such that for all $(g,\pi) \in \mathcal{D}^{k}(l^\infty)$
        $$T(g,\pi) = \int_E g(x) \theta(x) \; \left\langle \bigwedge\nolimits_{k} d^E_x \pi, \tau(x)\right\rangle \; d\Ha^k(x),$$
        where $\left\langle \bigwedge\nolimits_{k} d^E_x \pi, \tau(x)\right\rangle$ denotes the standard duality pairing between $k$-vectors and $k$-covectors. We write $T= [E,\theta,\tau]$. This representation can be seen as a measure theoretic version of a current $\bb{M}$ induced by a closed, oriented Riemannian manifold $M$. Indeed, we have $\bb{M}= [M,\theta,\tau]$ with $\theta = 1$ and $\tau$ is the usual orientation of $M$. 

\medskip

        With this representation of an integral current $T\in  \bI_k(X)$, we can now give an explicit description of the slices of $T$. Let $f\colon E \to \R$ be Lipschitz. Then, for almost all $t\in \R$, the preimage $E_t = E\cap f^{-1}(t)$ is $(k-1)$-rectifiable and $\textup{Tan}(E_t,x) = \ker(d_x^Ef)$ for $\Ha^{k-1}$-almost all $x \in E_t$. By \cite[Theorem 9.1]{ambrosio-kirchheim-2000}, for almost all $t\in \R$ and $\Ha^{k-1}$-almost every $x \in E_t$, we can write $\tau(x) = \xi(x)\wedge \tau_t(x)$ such that $\tau_t$ is an orientation of $E_t$ and 
        \begin{equation}\label{eq: repres-slice}
            \langle T, f,t \rangle = [ E_t, \theta, \tau_t]
        \end{equation}
        The orientation  $\tau_t$ is characterized by the following equality
        \begin{equation}\label{eq: slice-represent-1}
            \left\langle \bigwedge\nolimits_{k-1} d_x^{E_t}  \pi, \tau_t(x) \right\rangle \frac{\mathbf{J}_k(d_x^E(\pi,f))}{\mathbf{J}_{k-1}(d_x^{E_t}\pi)} = \left\langle \bigwedge\nolimits_{k} d^E_x( \pi,f), \tau(x)\right\rangle 
        \end{equation}
        for $\Ha^{k-1}$-almost all $x\in E_t$ and every Lipschitz map $\pi\colon E \to \R^{k-1}$. Here, the Jacobian is equal to 
        $$\mathbf{J}_k(d_x^E(\pi,f))= \left|\left\langle \bigwedge\nolimits_{k} d^E_x( \pi,f), \tau(x)\right\rangle \right|.$$
        %for almost all $t\in \R$ and $\Ha^{k-1}$-almost all $x \in E_t$. 
        A similar equality is true for $\mathbf{J}_{k-1}(d_x^{E_t}\pi)$, but we will not need it here.
        
        \medskip
        
        The next goal is to define the slice for a cycle $T\in \bI_k(X)$ with respect to a Lipschitz map $g\colon X\to \mathbb{S}^1$. Let $p \in \mathbb{S}^1$ and let $\varphi\colon X \to \R,\; x \to d_{\mathbb{S}^1}(g(x),p)$. We have 
        $$\langle T, \varphi,t\rangle =  \partial(T\on\{\varphi<t\}) = \partial(T \on g^{-1}(B(p,t)))$$
        and 
        $$\spt\left( \langle T, \varphi,t\rangle \right) \subset g^{-1}(\{\varphi = t\})$$
        for almost all $t \in \R$; see \cite[Theorem 5.6]{ambrosio-kirchheim-2000}. We conclude that, for a given differentiable chart $\psi$ of $\mathbb{S}^1$, we can find an open set $U\subset \mathbb{S}^1$ in the domain of $\psi$ with the following properties: $\partial U $ has zero $\Ha^1$-measure, $T \on V \in \bI_k(X)$ for $V=g^{-1}(U)$ and $\spt(\partial (T\on V)) \subset g^{-1}(\partial U)$. We fix an orientation of $\mathbb{S}^1$ and for $p\in \mathbb{S}^1$, we define
        \begin{equation}\label{eq: def-slice-S1}
            \langle T,g,p\rangle = \langle T\on g^{-1}(U),\psi\circ g,\psi(p)\rangle
        \end{equation}
        whenever the right hand side makes sense and $\psi\colon U\subset \mathbb{S}^1\to \R$ is any positively oriented chart as above. Using \cite[Theorem 5.6]{ambrosio-kirchheim-2000}, we get
        $$\spt (\partial \langle T,g,p\rangle) = \spt  (\langle \partial (T\on V),\psi\circ g,\psi(p)\rangle)  \subset g^{-1}(\partial U) \cap g^{-1}(p) $$
        for almost all $p \in U$ and in this case, $\langle T,g,p\rangle$ is a cycle. We claim that the definition does not depend on the choice of the chart. Letting $T = [E,\theta,\tau]$, then \eqref{eq: slice-represent-1} gives that for almost all $p \in U$
        $$\langle T,g,p\rangle = \langle T\on V,\psi\circ g,\psi(p)\rangle = [E\cap V \cap g^{-1}(p),\theta,\tau_p] = [E_p,\theta,\tau_p],$$
        where $E_p = E\cap g^{-1}(p)$ and $V=g^{-1}(U)$. In particular, $E_p$ does not depend on $\psi$. It remains to show that the orientation $\tau_p$ is also independent of $\psi$. Let $\varphi\colon U \to \R$ be another positively oriented chart. A direct computation using the chain rule (which also holds for the tangential differential) shows that for every Lipschitz map $\pi\colon E \to \R^{k-1}$, almost all $p\in \mathbb{S}^1$ and $\Ha^{k-1}$-almost all $x \in E_p$
        $$\left\langle \bigwedge\nolimits_{k} d_x^{E}  (\pi, \psi\circ g), \tau \right\rangle = d_{\varphi(g(x))}(\psi\circ \varphi^{-1})  \left\langle \bigwedge\nolimits_{k} d_x^{E}  (\pi, \varphi\circ g), \tau \right\rangle  $$
        as well as
        $$\mathbf{J}_k(d_x^E (\pi,\psi\circ g)) = |d_{\varphi(g(x))}(\psi\circ \varphi^{-1})| \mathbf{J}_k(d_x^E (\pi,\varphi\circ g)).$$
        Since $\psi$ and $\varphi$ are positively oriented, we have $d_{\varphi(g(x))}(\psi\circ \varphi^{-1})>0$ for every $x\in V$. Therefore, \eqref{eq: slice-represent-1} implies that the orientation $\tau_p$ does not depend on the chart $\psi$. Hence, for almost all $p \in \mathbb{S}^1$, the slice $\langle T,g,p\rangle$ is a well-defined integral $(k-1)$-cycle and is independent of the chart used in \eqref{eq: def-slice-S1}. Notice that the ambient structure of $l^\infty$ and the representation $T=[E,\theta,\tau]$ are only used to show that the slice $\langle T,g,p\rangle$ is well-defined for almost all $p \in \mathbb{S}^1$. Moreover, by \cite[Theorem 9.6]{ambrosio-kirchheim-2000}, the definition of the slice is intrinsic in the sense that it does not depend on the isometric embedding of $X$ into $l^\infty$. Thus, we use the slice for Lipschitz maps $g\colon X \to \mathbb{S}^1$ in separable metric spaces without embedding them into $l^\infty$. We conclude the section with the following lemma.

        \begin{lemma}\label{lemma: pushfw-slice-is-slice-phsfwd}
            Let $\varphi \colon X \to Y$ be a Lipschitz map between two separable metric spaces. Furthermore, let $T \in \bI_k(X)$ be a cycle and $g\colon Y\to\mathbb{S}^1$ be Lipschitz. Then 
            $$\varphi_\#\langle T, g\circ \varphi, p \rangle = \langle \varphi_\# T, g, p \rangle$$
            for almost all $p \in \mathbb{S}^1$.
        \end{lemma}

        \begin{proof}
            Let $\psi\colon U \subset \mathbb{S}^1\to \R$ be a positively oriented chart such that $\partial U$ has zero $\Ha^1$-measure, $T\on W\in \bI_k(X)$ and $(\varphi_\#T)\on V\in \bI_k(Y)$ for $W= (g\circ\varphi)^{-1}(U)$ and $V = g^{-1}(U)$, respectively. Recall that by \cite[Theorem 5.6]{ambrosio-kirchheim-2000} we can find, for each chart $\psi$, an open set in its domain that has these properties. It follows from the definition \eqref{eq: def-slice-S1} and Lemma \ref{lemma: pushfw-slice-is-slice-phsfwd-normal} that
            $$\varphi_\#\langle T,g\circ\varphi,p\rangle = \varphi_\#\langle T\on W,\psi\circ g\circ\varphi,\psi(p)\rangle = \langle \varphi_\#(T\on W),\psi\circ g,\psi(p)\rangle$$
            for almost all $p \in \mathbb{S}^{1}$. On the other hand,
            $$\langle \varphi_\# T , g,p\rangle = \langle (\varphi_\#T)\on V,\psi\circ g,\psi(p)\rangle$$
            for almost all $p \in \mathbb{S}^1$. For $(f,\pi)\in \mathcal{D}^k(Y)$ we have
            $$(\varphi_\#T)\on V(f,\pi) = T((\mathbbm{1}_V\circ\varphi ) (f\circ\varphi),\pi\circ\varphi) = T(\mathbbm{1}_{\varphi^{-1}(V)} f\circ\varphi,\pi\circ\varphi)$$
            and
            $$ \varphi_\#(T\on W)(f,\pi)= T(\mathbbm{1}_{W} f\circ\varphi,\pi\circ\varphi).$$
            Since $W = \varphi^{-1}(V)$ we have $(\varphi_\#T)\on V=\varphi_\#(T\on W)$ and in particular,
            $$\varphi_\#\langle T, g\circ \varphi, p \rangle = \langle \varphi_\# T, g, p \rangle$$
            for almost all $p \in \mathbb{S}^1$. This completes the proof.
        \end{proof}

    \subsection{Homology}\label{sec: homology}
        Next, we discuss singular homology and homology groups via integral currents. We refer to \cite{Hatcher} and \cite{Mitsuishi} for more details. Let $X$ be a complete metric space and let $k\geq0$ be an integer. The standard $k$-dimensional Euclidean simplex is denoted by $\Delta^k$. We call a continuous map $\varphi \colon \Delta^k \to X$ a singular $k$-simplex. A singular $k$-chain $c$ is a formal sum
        $$c = \sum_{i=1}^N n_i \cdot \varphi_i,$$
        where each $\varphi_i$ is a singular $k$-simplex and $n_i \in \Z$. The boundary of $c$ is the $(k-1)$-chain $bc$ defined as
        $$bc = \sum_{i=1}^N n_i \cdot b\varphi_i,$$
        where $b\varphi_i = \sum_j (-1)^j \varphi_{i,j}$ and the $\varphi_{i,j}$ are $\varphi_i$ restricted to the different boundary components of $\Delta^k$. At this point we fixed an enumeration of the boundary components for each $\Delta^k$. We say $c$ is a cycle if $bc = 0$. The space of singular $k$-chains in $X$ is denoted by $C_k(X)$. It follows that these groups, together with the boundary operator, form a chain complex. We write $H_k(X)$ for the $k$th homology group of this chain complex and call it the $k$th singular homology group of $X$. Two singular $k$-cycles $c$ and $c'$ that represent the same class in $H_k(X)$ are said to be homologous, that is, there exists $v\in C_{k+1}(X)$ such that $bv= c-c'$. Let $f\colon X \to Y$ be continuous and let $c= \sum_{i=1}^N n_i \cdot \varphi_i$ be a singular $k$-chain. The pushforward $f_\#c$ of $c$ by $f$ is the singular $k$-chain in $Y$ given by 
        $$f_\#c = \sum_{i=1}^N n_i \cdot (f\circ \varphi_i).$$
        It is not difficult to see that this induces a homomorphism at the level of homology $f_*\colon H_k(X) \to H_k(Y)$. In particular, $f_*[c] = [f_\#c]$ for all $c\in C_k(X)$, where $[c]$ and $ [f_\#c]$ are the homology classes of $c$ and $f_\#c$, respectively. If two continuous maps $f,g \colon X\to Y$ are homotopic, then they induce the same homomorphism $H_k(X)\to H_k(Y)$. 
        
        \medskip

        Given $T\in \bI_k(X)$, the boundary-rectifiability theorem of Ambrosio and Kirchheim \cite[Theorem 8.6]{ambrosio-kirchheim-2000} implies that $\partial T\in \bI_{k-1}(X)$. Therefore, the boundary operator $\partial$ for integral currents maps $\bI_{k}(X)$ into $\bI_{k-1}(X)$ and in particular, 
        $$
        \dotsm \overset{\partial_{k+1}}{\longrightarrow} \bI_k(X) \overset{\partial_{k}}{\longrightarrow} \bI_{k-1}(X) \overset{\partial_{k-1}}{\longrightarrow} \dotsm \overset{\partial_1}{\longrightarrow} \bI_0(X) 
        $$
        is a chain complex. We denote by $H_k^\textup{IC}(X)$ the $k$th homology group of this chain complex and call $H_k^\textup{IC}(X)$ the $k$th homology group via integral currents. Two cycles $S,T \in \bI_{k}(X)$ have the same homology class $[T] = [S]$ if and only if there exists $V\in \bI_{k+1}(X)$ with $\partial V = S-T$. Let $f\colon X \to Y$ be Lipschitz. The boundary operator and pushforward commute, that is, $\partial (f_\# T) = f_\# (\partial T)$ for every $T\in \bI_k(X)$. Hence, $f$ induces a homomorphism $f_*\colon H_k^\textup{IC}(X) \to H_k^\textup{IC}(Y)$. As in the case of singular homology groups, two Lipschitz maps $f,g\colon X \to Y$ induce the same homomorphism $H_k^\textup{IC}(X) \to H_k^\textup{IC}(Y)$ if there exists a Lipschitz homotopy between $f$ and $g$.

        \medskip
        
        A singular $k$-chain $c$ is called a singular Lipschitz $k$-chain if each singular simplex of $c$ is a Lipschitz map. The space of singular Lipschitz $k$-chains in $X$ is denoted by $C_k^\textup{Lip}(X)$. The boundary of a singular Lipschitz chain is again a Lipschitz chain. Thus, we can define the singular Lipschitz homology groups $H_k^\textup{Lip}(X)$ via these Lipschitz chains. Notice that every singular Lipschitz $k$-chain induces an integral $k$-current as well as a singular $k$-chain. It follows that the inclusions $C_k^\textup{Lip}(X)\hookrightarrow \bI_k(X)$ and $C_k^\textup{Lip}(X)\hookrightarrow C_k(X)$, respectively, are chain maps. It is well known that if the underlying space is locally Lipschitz contractible, then these inclusion maps induce isomorphisms on the level of homology; see \cite{Mitsuishi,riedweg2009singular}.

        \begin{theorem}\label{thme: sing-integral-homology-coincide}
            Let $X$ be a compact metric space. Suppose that there exist $R, \lambda>0$ such that every ball $B$ in $X$ of radius $0<r<R$ is contractible via a Lipschitz map $\varphi\colon [0,1] \times B\to X$ satisfying
                $$d(\varphi(s,x),\varphi(t,y)) \leq \lambda r |s-t| + \lambda d(x,y)$$
            for all $s,t \in [0,1]$ and every $x,y\in B$. Then, 
            $$H_k^\textup{IC}(X) \cong H_k^\textup{Lip}(X) \cong H_k(X)$$
            for all $k \in \N$. The isomorphisms are induced by the inclusions $C_k^\textup{Lip}(X)\hookrightarrow \bI_k(X)$ and $C_k^\textup{Lip}(X)\hookrightarrow C_k(X)$, respectively.
        \end{theorem}

        Clearly, every closed Riemannian manifold satisfies the conditions of the theorem, as it is locally bi-Lipschitz to Euclidean space.

    \subsection{Isoperimetric inequalities}\label{see: isoper}
        Let $X$ be a complete metric space and let $T \in \bI_1(X)$ be an integral $1$-cycle. An integral $2$-current $V\in \bI_{2}(X)$ is called a filling of $T$ if $\partial V = T$. We say $X$ satisfies a $1$-dimensional (Euclidean) small mass isoperimetric inequality with constants $C, \Lambda>0$ if the following holds. For every cycle $T \in \bI_{1}(X)$ with $\bM(T) < \Lambda$, there exists a filling $V \in \bI_{2}(X)$ of $T$ satisfying $$\bM(V) \leq C  \bM(T)^2.$$

        \medskip

        Next, we introduce the local isoperimetric inequality for Lipschitz curves and compare it to the isoperimetric inequality for integral currents. The space of Sobolev maps from the $2$-dimensional unit disk $\mathbb{D}$ into $X$ is denoted by $W^{1,2}(\mathbb{D},X)$. For a detailed exposition of the theory on Sobolev maps in metric spaces, see \cite{Heinonen-Koskela-Shanmugalingam-Tyson-2015}. Let $\varphi \in W^{1,2}(\mathbb{D},X)$. Then, for almost all $x \in \mathbb{D}$, the map $\varphi$ has an approximate metric derivative $\apmd_x \varphi$ at $x$, which is a seminorm on $\R^2$. We need two different Jacobians for a seminorm $s$ on $\R^2$. If $s$ defines a norm, then we put $\mathbf{J}(s) = \Ha^2_{(\R^2,s)}([0,1]^2)$ and $\mathbf{J}^*(s) = 4/ \leb^2(P)$, where $\leb^2(P)$ is the Lebesgue measure of the parallelogram of smallest area containing the unit ball with respect to $s$. In case $s$ is degenerate, we set $\mathbf{J}(s) = \mathbf{J}^*(s) = 0$. We note that
        $ \mathbf{J}^*(s) \leq 2 \mathbf{J}(s)$ for every norm $s$ on $\R^2$ \cite[Lemma 9.2]{ambrosio-kirchheim-2000}. The (parameterized) Hausdorff area of $\varphi$ is defined as 
        $$\area (\varphi) \coloneqq \int_{\mathbb{D}}  \mathbf{J}(\apmd_x \varphi) \; d \leb^2(x).$$
        Analogously, the (parameterized) Gromov $\textup{mass}_*$ area, denoted by $\area^*(\varphi)$, is defined by replacing $\mathbf{J}$ with $\mathbf{J}^*$ in the previous definition. For a subset $A \subset \mathbb{D}$, the area $\area(\varphi|_A)$ of $\varphi$ restricted to $A$ is defined by integrating only over $A$. We refer to \cite{alvarez-thompson04, Lytchack-wenger-Canonical-parameterizations} for more information on Jacobians and area. For almost all $v \in \mathbb{S}^1$, the curve $t\mapsto \varphi(tv)$ is absolutely continuous on $[1/2,1)$. The trace of $\varphi$ is given by
        $$\trace(\varphi) (v)= \lim_{t \nearrow 1} \varphi(vt)$$
        for almost every $v \in \mathbb{S}^1$. We say $X$ satisfies a local quadratic isoperimetric inequality for Lipschitz curves with constants $D,\Gamma >0$ if the following holds. For every Lipschitz curve $\gamma\colon \mathbb{S}^1 \to X$ with $\textup{length}(\gamma) < 	\Gamma $, there exists  $\varphi\in W^{1,2}(\mathbb{D},X)$ satisfying $\trace (\varphi) = \gamma$ and $\textup{Area}(\varphi) \leq D \len( \gamma)^2$. If this holds for all Lipschitz curves of arbitrary length, then we say $X$ satisfies a quadratic isoperimetric inequality for Lipschitz curves with constant $D$. It is well-known to experts that a local quadratic isoperimetric inequality for Lipschitz curves implies a $1$-dimensional small mass isoperimetric inequality for integral currents. However, we could not find a reference and thus, we provide a proof. 

        \begin{proposition}\label{prop: sobolev-isop-implies-current-isoper}
            Let $X$ be a complete and separable metric space. Suppose that $X$ satisfies a local quadratic isoperimetric inequality for Lipschitz curves with constants $D,	\Gamma >0$. Then, $X$ satisfies a $1$-dimensional small mass isoperimetric inequality for integral currents with constants $2 D,\Gamma$. 
        \end{proposition}

        The main idea of the proof is to approximate a given Sobolev map $\varphi\in W^{1,2}(\mathbb{D},X)$ by Lipschitz maps with uniformly bounded area, similar as in \cite[Proposition 3.1]{Lytchak-wenger-robert}. We can construct a current in $X$ using these approximations and a pushforward argument. This approach has also been used in \cite{basso2023geometric, ikonen2023pushforward, denis-non-orientable}.

        \begin{proof}
            Let $\gamma\colon \mathbb{S}^1 \to X$ be Lipschitz. Suppose that there exists a Sobolev map $\varphi \colon \mathbb{D}\to X$ with $\trace (\varphi) = \gamma$. We claim that there exists $U \in \bI_2(X)$ satisfying $\partial U = \gamma_\#\bb{\mathbb{S}^1}$ and $\bM(U) \leq 2 \area(\varphi)$. Let $B = B(0,2)$ be the open ball of radius $2$ in $\R^2$. We extend $\varphi$ to the closed ball $\overline{B}$ as follows
            $$
            \Tilde{\varphi}(v) =
            \begin{cases}
            \varphi(v), &v \in \mathbb{D}\\
            \varphi\big( \frac{v}{\norm{v}}\big),  &v \notin \mathbb{D}.
            \end{cases}
            $$
            Here, $\norm{\cdot}$ denotes the standard Euclidean norm. Notice that $\Tilde{\varphi}$ belongs to $W^{1,2}(B,X)$ and $\area(\Tilde{\varphi})=\area(\varphi)$. It follows that there exists $g \in L^2(B)$ such that 
            \begin{equation}\label{eq: sobolev-isop-implies-current-isoper}
                d(\Tilde{\varphi}(x),\Tilde{\varphi}(y) )\leq \norm{x-y} (g(x)+g(y))
            \end{equation}
            for almost all $x,y \in B$; see \cite[Theorem 8.1.7]{Heinonen-Koskela-Shanmugalingam-Tyson-2015}. For $k \in \N$, we define $A_k = \{ x \in B \colon g(x) \leq k\}$ and
            $$\varepsilon_k = \int_{B \setminus A_k} g(x)^2 \; d\leb^2(x).$$
            Since $g \in L^2(B)$, Chebyshev’s inequality and the absolute continuity of integrals imply that $\leb^2(B\setminus A_k) \leq \frac{\varepsilon_k}{k^2}$ and $\varepsilon_k \to 0$ as $k \to \infty$. In particular, each $A_k$ is $\frac{\sqrt{\varepsilon_k}}{k}$-dense in $B$. It follows from \eqref{eq: sobolev-isop-implies-current-isoper} that the map $\Tilde{\varphi}$ is $2k$-Lipschitz on each $A_k$. Moreover, by construction, the map $\Tilde{\varphi}$ is Lipschitz on $\overline{B}\setminus \overline{B(0,1)}$. Therefore, for $k\in \N$ sufficiently large, $\Tilde{\varphi}$ is $4k$-Lipschitz on $B_k = A_k \cup \big( \overline{B}\setminus (\overline{B(0,\frac{3}{2})}\big)$. To simplify the notation, we assume that $\Tilde{\varphi}$ is $4k$-Lipschitz on $B_k$ for all $k \in \N$. Embed $X$ into $l^\infty$. For each $k \in \N$, let $\varphi_k \colon \overline{B} \to l^\infty$ be a $4k$-Lipschitz extension of $\Tilde{\varphi}|_{B_k}$. We define a sequence $U_k = \varphi_{k\#} \bb{B} \in \bI_2(l^\infty)$. We have
            $$\bM(\varphi_{k\#}\bb{B\setminus B_k}) \leq \Lip(\varphi_k)^2\cdot \leb^2(B\setminus B_k) \leq 16 \varepsilon_k,$$
            and
            $$\bM(\varphi_{k\#}\bb{B_k}) \leq \int_{B_k} \mathbf{J}^*(\apmd_x \varphi_k) \; d\leb^2(x) = \area^*(\varphi|_{B_k})$$
            for all $k \in \N$. Therefore,
            $$\bM(U_k) \leq  \bM(\varphi_{k\#}\bb{B_k}) +\bM(\varphi_{k\#}\bb{B\setminus B_k}) \leq \area^*(\varphi) + 16 \varepsilon_k$$
            for each $k \in \N$. Furthermore, $\varphi_k(2v) = \gamma(v)$ for all $k \in \N$ and almost every $v \in \mathbb{S}^1$, and hence $\partial U_k = \varphi_{k\#}(\partial \bb{B})= \gamma_\#\bb{\mathbb{S}^1}$ for each $k \in \N$. In particular, $\bN(U_k)$ is uniformly bounded. It follows from the compactness and closure theorem for integral currents \cite[Theorem 5.2 and Theorem 8.5]{ambrosio-kirchheim-2000} that there exists $U \in \bI_2(l^\infty)$ such that the $U_k$ converge weakly to $U$. Clearly, $\partial U = \gamma_\#\bb{\mathbb{S}^1}$. Since the $B_k$ are $\frac{\sqrt{\varepsilon_k}}{k}$-dense in $\overline{B}$ and each $\varphi_k$ is $4k$-Lipschitz, we have $\spt U_k \subset \varphi_k(\overline{B}) \subset N_{4\sqrt{\varepsilon_k}}^{l^\infty}(X)$ for all $k\in \N$. Thus, $U$ belongs to $\bI_2(X)$. Finally, the lower semicontinuity of mass implies
            $$\bM(U) \leq \liminf_{k\to \infty} \bM(U_k) \leq \area^*(\varphi)\leq 2 \area(\varphi).$$
            This proves the claim.

            \medskip

            Now, let $T \in \bI_1(X)$ be such that $\bM(T) <\Gamma$. By \cite[Theorem 5.3]{Enrico-Decomposition}, there exist countably many injective Lipschitz loops $\gamma_k \colon \mathbb{S}^1 \to X$ such that
            \begin{equation}\label{eq: sobolev-isop-implies-current-isoper-2}
               T = \sum_{k=1}^\infty \gamma_{k\#}\bb{\mathbb{S}^1} \quad\textup{ and }\quad \bM(T) = \sum_{k=1}^\infty \bM\big(\gamma_{k\#}\bb{\mathbb{S}^1}\big). 
            \end{equation}
            We conclude that $\bM\big(\gamma_{k\#}\bb{\mathbb{S}^1}\big)= \len(\gamma_k) < \Gamma$ for each $k \in \N$. Because $X$ satisfies a local quadratic isoperimetric inequality for Lipschitz curves with constants $D,\Gamma >0$, there exist countably many $\varphi_k \in W^{1,2}(\mathbb{D},X)$ satisfying $\trace(\varphi_k) = \gamma_k$ and $\area(\varphi_k) \leq D \len(\gamma_k)^2 = D\bM\big(\gamma_{k\#}\bb{\mathbb{S}^1}\big)^2$ for each $k \in \N$. It follows from the claim that there exist countably many $U_k\in \bI_2(X)$ with $\partial U_k = \gamma_{k\#}\bb{\mathbb{S}^1}$ and $\bM(U_k)\leq 2D \bM(\gamma_{k\#}\bb{\mathbb{S}^1})^2$ for every $k \in \N$. For $m \in \N$, we define $V_m = \sum_{k=1}^m U_k \in \bI_2(X)$. The second part of \eqref{eq: sobolev-isop-implies-current-isoper-2} implies that the sequence $(V_m)_m$ is Cauchy in $\bI_2(X)$. We denote the limit by $V \in \bI_2(X)$. We have $\partial V = T$. Finally, by the convexity of $x \mapsto x^2$, 
            \begin{align*}
                \bM(V) &\leq \lim_{m\to \infty} \sum_{k=1}^m \bM(U_k) \leq 2 D \lim_{m\to \infty} \sum_{k=1}^m \bM(\gamma_{k\#}\bb{\mathbb{S}^1})^2
                \\
                &\leq 2 D \lim_{m\to \infty} \left(\sum_{k=1}^m \bM(\gamma_{k\#}\bb{\mathbb{S}^1})\right)^2 = 2 D \bM(T)^2.
            \end{align*}
            This completes the proof. 
        \end{proof}

        For the remainder of the section, let $X$ be a metric space that is homeomorphic to a closed, oriented, connected smooth surface. Furthermore, we suppose that $X$ is Ahlfors $2$-regular and linearly locally contractible. It was shown in \cite{semmes-1996} that $X$ supports a weak 1-Poincar\'e inequality; see also \cite[Corollary 1.6]{basso2023geometric}. It is well-known that this implies that $X$ is quasiconvex; see e.g. \cite[Theorem 8.3.2]{Heinonen-Koskela-Shanmugalingam-Tyson-2015}. Poincar\'e inequalities are strongly related to relative isoperimetric inequalities. More precisely, by \cite[Theorem 1.1]{korte} there exist $D,\lambda\geq 1$ such that
        \begin{equation}\label{eq: relative-isoper-ineq}
            \min\{\Ha^2(B\cap E), \Ha^2(B\setminus E) \} \leq D \cdot \Ha^{1}(\lambda B \cap \partial E)^2
        \end{equation}
        for every Borel set $E \subset X$ and every ball $B\subset X$. Here, $\lambda B$ denotes the ball with the same center as $B$ and radius equal to $\lambda$ times the radius of $B$. Notice that in \cite{korte} the inequality \eqref{eq: relative-isoper-ineq} is formulated for the Hausdorff measure of codimension one $\Ha$ defined as
        $$\Ha(E) = \lim_{R\to 0} \inf\left\{\sum_{i=0}^\infty \frac{\Ha^2(B(x_i,r_i))}{r_i} \colon E \subset \bigcup_{i=1}^\infty B(x_i,r_i), r_i \leq R \right\}$$
        for all Borel sets $E \subset X$. Since $X$ is Ahlfors $2$-regular, we have $\Ha \leq c \cdot \Ha^1$, where $c>0$ denotes the Ahlfors $2$-regularity constant of $X$. Thus, the relative isoperimetric inequality as in \eqref{eq: relative-isoper-ineq} follows.

        \begin{theorem}\label{thme: small-mass-isop-ahlfors-llc}
            Let $X$ be a metric space that is homeomorphic to a closed, oriented, connected smooth surface. Suppose that $X$ is Ahlfors $2$-regular and linearly locally contractible. Then, $X$ satisfies a local quadratic isoperimetric inequality for Lipschitz curves.
        \end{theorem}

        \begin{proof}
            Let $Y \subset X$ be a subset that is homeomorphic to the closed disk $\overline{\mathbb{D}}$ with $\diam Y < \diam (X)/3$ and let $\Omega \subset Y$ be a Jordan domain. A Jordan domain $\Omega$ is an open set homeomorphic to $\mathbb{D}$ such that $\partial\Omega$ is homeomorphic to $\mathbb{S}^1$. The relative isoperimetric inequality \eqref{eq: relative-isoper-ineq} implies that
            $$\Ha^2(\Omega) \leq D \cdot \Ha^1(\partial \Omega)^2 = D \cdot \len(\partial \Omega)^2.$$
            It follows from \cite[Theorem 1.4]{Lytchack-wenger-Canonical-parameterizations} that $Y$ satisfies a quadratic isoperimetric inequality for Lipschitz curves with constant $D$. Covering $X$ by finitely many subsets $Y$ as above, we conclude that $X$ satisfies a local quadratic isoperimetric inequality for Lipschitz curves.
        \end{proof}

\section{Codimension one homology}    
     Throughout this section, let $X$ be a metric space with finite Hausdorff $n$-measure that is homeomorphic to a closed, oriented, connected Riemannian $n$-manifold $M$. We denote by $\varrho\colon X \to M$ a homeomorphism of degree one. It is a classical fact that every primitive codimension one singular homology class in $M$ can be represented by an embedded submanifold. A homology class is called primitive if it is the zero class or if it is not a nontrivial multiple of another class. We need a version of this statement for the homology via integral currents.

    \begin{lemma}\label{lemma: topog-repr-homo-class-map-into-S1}
        Let $[\eta] \in H_{n-1}^\textup{IC}(M)$ be primitive and non-zero. Then, there exists a surjective smooth map $f\colon M \to  \mathbb{S}^1$ with the following property. For almost all $p \in \mathbb{S}^1$, the preimage $f^{-1}(p) = N_p$ is a closed, oriented smooth $(n-1)-$manifold and $\big[ \bb{N_p} \big] =  [\eta] = [\langle \bb{M},f,p\rangle]$.
    \end{lemma}

    Here, $\langle \bb{M},f,p\rangle$ is the slice of the fundamental class $\bb{M}$ of $M$ by $f$ at $p\in \mathbb{S}^1$ as defined in Section \ref{sec: currents}, and $\bb{N_p} \in \bI_{n-1}(M)$ denotes the integral $(n-1)$-cycle induced by $N_p$.

    \begin{proof}
        The homology groups $H_{n-1}^\textup{IC}(M)$ and $H_{n-1}(M)$ are isomorphic; see Theorem \ref{thme: sing-integral-homology-coincide}. By abuse of notation, we also denote by $[\eta]$ the singular homology class given by the isomorphism $H_{n-1}(M) \to H_{n-1}^\textup{IC}(M)$. 
        %It follows from the universal coefficient theorem that $H_{n-1}(M) \cong \textup{Hom}(H_1(M),\Z)$. Since $\Z$ is abelian we further have $\textup{Hom}(H_1(M),\Z) \cong \textup{Hom}(\pi_1(M),\Z)$. It is not difficult to see that the elements in $\textup{Hom}(\pi_1(M),\Z)$ are induced by continuous maps $f\colon M \to \mathbb{S}^1$. In particular, $H_{n-1}(M) \cong [M,\mathbb{S}^1]$, where $[M,\mathbb{S}^1]$ denotes the set of homotopy classes $M \to \mathbb{S}^1$. We fix a smooth representative $f \colon M \to \mathbb{S}^1$ of $[\eta]\in H_{n-1}(M)$. Let $t \in \mathbb{S}^1$ be a regular value such that the preimage $f^{-1}(t) = N_t$ is a closed, oriented smooth submanifold of $M$. The Poincar\`e duality implies that the homology class in $H_{n-1}(M)$ induced by the fundamental class of $N_t$ coincides with $[\eta]$; see e.g. \cite{Meeks-codimension-one-homol}. 
        It follows from the Poincar\'e duality that $[\eta]$ can be represented by a surjective smooth map $f\colon M \to \mathbb{S}^1$. The standard duality theorems imply that for each regular value $p \in \mathbb{S}^1$, the preimage $f^{-1}(p) = N_p$ is a closed, oriented submanifold representing $[\eta]$; see \cite{Meeks-codimension-one-homol}. Using a triangulation of $N_p$ and Theorem \ref{thme: sing-integral-homology-coincide}, one easily checks that the singular homology class induced by $N_p$ and the homology class in $H_{n-1}^{IC}(M)$ induced by $\bb{N_p}$ agree under the isomorphism $H_{n-1}(M) \to H_{n-1}^\textup{IC}(M)$. Now, let $p\in \mathbb{S}^1$ be a regular value such that the slice $\langle \bb{M},f,p\rangle$ exists. Then, $\langle \bb{M},f,p\rangle$ has a representation $[M \cap f^{-1}(p),\theta,\tau_p]$ with $\theta = 1$ and $\tau_p$ is an orientation of $M\cap f^{-1}(p) = N_p$ induced by $f$ and the orientation of $M$; see \eqref{eq: repres-slice}. It follows from the proof of \cite[Theorem 9.7]{ambrosio-kirchheim-2000} that $\tau_p$ is the same orientation as the pullback orientation of $N_p$ given by $f$; see \cite[p. 100]{Guillemin-Pollack-Differential-topology}. Therefore, $\langle \bb{M},f,p\rangle = [N_p, \theta, \tau_p] = \bb{N_p}$. In particular, $[ \bb{N_p} ] =  [\eta] = [\langle \bb{M},f,p\rangle]$ for almost all $p \in \mathbb{S}^1$. This completes the proof. 
    \end{proof}

    Theorem \ref{thme: codimension-one-homology} will be a direct consequence of the following proposition.

    \begin{proposition}\label{prop: induced-homo-integral-homology-codim-1}
        Suppose that $X$ has a metric fundamental class. Then, there exists $\varepsilon>0$ such that every Lipschitz $\varphi\colon X \to M$ satisfying $d(\varrho,\varphi) < \varepsilon$ induces a surjective homomorphism $\varphi_*\colon H_{n-1}^\textup{IC}(X)\to H_{n-1}^\textup{IC}(M)$. 
    \end{proposition}

    Recall that a metric fundamental class of $X$ is a cycle $T \in \bI_n(X)$ satisfying $\varphi_\#T = \deg(\varphi) \cdot \bb{M}$ for every Lipschitz map $\varphi\colon X \to M$, where $\deg(\varphi)$ denotes the (topological) degree of $\varphi$. In fact, $T$ also satisfies a second property that we do not need here; see Definition \ref{def: metric fundamental class}.

    \begin{proof}[Proof of Proposition \ref{prop: induced-homo-integral-homology-codim-1}]        
        Since $M$ is a closed Riemannian manifold, there exists a bi-Lipschitz embedding $\iota\colon M \xhookrightarrow{} \R^N$ and a Lipschitz retraction $\pi \colon N_{2\varepsilon}(\iota(M)) \to \iota(M)$ for some $N\geq n$ and $\varepsilon>0$. See e.g. \cite[Theorem 3.10]{Heinonen-geom-embeddings} and \cite[Theorem 3.1]{hohti-1993}. Let $\varphi,\psi \colon X \to M$ be two Lipschitz maps satisfying $d(\varphi,\varrho),d(\psi,\varrho)<\varepsilon$. The straight-line homotopy $H$ between $(\iota \circ\varphi)$ and $(\iota \circ\psi)$ stays inside $N_{2\varepsilon}(\iota(M))$. Therefore, the composition $(\iota^{-1}\circ \pi\circ H)\colon [0,1]\times X \to M$ is well-defined, Lipschitz, and a homotopy between $\varphi$ and $\psi$. We conclude that $\varphi_*$ and $\psi_*$ are equal as homomorphism $ H_{n-1}^\textup{IC}(X)\to H_{n-1}^\textup{IC}(M)$.

        \medskip
        The homology group $H_{n-1}(M)$ is finitely generated because $M$ is an absolute neighborhood retract; see e.g. \cite[Corollary A.8]{Hatcher}. Thus, $H_{n-1}^\textup{IC}(M)$ is finitely generated as well. We fix a generating set of $H_{n-1}^\textup{IC}(M)$ that consists of primitive elements of $H_{n-1}^\textup{IC}(M)$. Clearly, it suffices to show that for each generator $[\eta] \in H_{n-1}^\textup{IC}(M)$, there exists $[\xi] \in H_{n-1}^\textup{IC}(X)$ with $\varphi_*[\xi] = [\eta]$ for all Lipschitz maps $\varphi\colon X \to M$ satisfying $d(\varrho,\varphi) < \varepsilon$. Let $[\eta] \in H_{n-1}^\textup{IC}(M)$ be a generator. It follows from Lemma \ref{lemma: topog-repr-homo-class-map-into-S1} that there exists a smooth map $f\colon M \to \mathbb{S}^1$ such that $ [\langle \bb{M},f,p\rangle]= [\eta]$ for almost all $p\in \mathbb{S}^1$. By Lemma \ref{lemma: lip-approx-easy}, we can approximate the homeomorphism $\varrho\colon X \to M$ by a Lipschitz map $\varphi\colon X\to M$ that is homotopic to $\varrho$, satisfies $d(\varphi,\varrho)<\varepsilon$ and such that the composition $f\circ \varphi$ is surjective. The degree is a homotopy invariant and thus, $\deg(\varphi) = \deg(\varrho) = 1$. As explained in Section \ref{sec: currents}, the slice $\langle T, g,p\rangle$ exists and is an integral $(n-1)$-cycle for almost all $p \in \mathbb{S}^1$. Lemma \ref{lemma: pushfw-slice-is-slice-phsfwd} implies 
        $$\varphi_*[\langle T,f\circ \varphi,p\rangle]=[\varphi_\#\langle T,f\circ \varphi,p\rangle] = [\langle \varphi_\# T, f,p\rangle]=  [\langle \bb{M}, f,p\rangle] = [\eta],$$
        for almost every $p\in \mathbb{S}^1$. Here, we used that $T$ is the metric fundamental class of $X$ and hence, $\varphi_\#T = \deg(\varphi) \cdot \bb{M} = \bb{M}$. This completes the proof. 
    \end{proof}

    We conclude the section with the short proof of Theorem \ref{thme: codimension-one-homology}

    \begin{proof}[Proof of Theorem \ref{thme: codimension-one-homology}]
        Let $\varepsilon>0$ be given by Proposition \ref{prop: induced-homo-integral-homology-codim-1}. Applying Lemma \ref{lemma: lip-approx-easy} to the homeomorphism $\varrho\colon X \to M$ we obtain a Lipschitz map $\varphi\colon X \to M$ with $d(\varphi,\varrho)<\varepsilon$. Therefore, the induced homomorphism $\varphi_*\colon H_{n-1}^\textup{IC}(X)\to H_{n-1}^\textup{IC}(M)$ is surjective. We denote by $\Theta \colon H_{n-1}^\textup{IC}(M) \to H_{n-1}(M)$ the isomorphism of Theorem \ref{thme: sing-integral-homology-coincide}. Since $\varrho$ is a homeomorphism, the induced map $\varrho^{-1}_*\colon H_{n-1}(M) \to H_{n-1}(X)$ is an isomorphism. We conclude that $(\varrho^{-1}_* \circ \Theta \circ \varphi_*)\colon H_{n-1}^\textup{IC}(X) \to H_{n-1}(X)$ is a surjection. 
    \end{proof}

\section{One-dimensional homology}\label{sec: one-dime}
        The goal of this section is to prove Theorem \ref{thme: one-homology}. Let $X$ be a quasiconvex metric space with finite Hausdorff $n$-measure that satisfies a $1$-dimensional small mass isoperimetric inequality for integral currents with constants $ C, \Lambda>0$. Furthermore, suppose that there exists a homeomorphism $\varrho\colon X \to M$, where $M$ is a closed, Riemannian manifold. Notice that $X$ still satisfies a $1$-dimensional small mass isoperimetric inequality for integral currents after a bi-Lipschitz change of the metric and the conclusion of Theorem \ref{thme: one-homology} is invariant under bi-Lipschitz changes of the metric. Thus, we may assume that $X$ is geodesic. 
        
        \medskip

        First, we show that every homology class in $H_1^\textup{IC}(X)$ has a representative given by a finite sum of injective Lipschitz loops.

        \begin{lemma}\label{lemma: integral-homology-repres-finite-loops-}
            Let $T\in \bI_1(X)$ be a cycle. Then, there exist finitely many injective Lipschitz loops $\gamma_i \colon \mathbb{S}^1 \to X$ and $V \in \bI_2(X)$ such that
            $$\partial V = T - \sum_{i=1}^N \gamma_{i\#}\bb{\mathbb{S}^1}.$$
        \end{lemma}

        \begin{proof}
            It follows from \cite[Theorem 5.3]{Enrico-Decomposition} that there exist countably many injective Lipschitz loops $\gamma_i \colon \mathbb{S}^1\to X$ such that 
            $$T = \sum_i^\infty  T_i\quad \textup{and} \quad \bM(T) = \sum_i^\infty \bM(T_i),$$
            where $T_i = \gamma_{i\#}\bb{\mathbb{S}^1}$. Let $N$ be sufficiently large such that for all $i>N$ we have $\bM(T_i) < \min({1,\Lambda})$. Then, for each $i>N$, there exists a filling $V_i \in \bI_2(X)$ of $T_i$ satisfying 
            $$\bM(V_i) \leq C\bM(T_i)^2 \leq C \bM(T_i).$$
            Therefore, the sum $V = \sum_{i>N} V_i$ converges and defines an integral $2$-current that is a filling of $\sum_{i>N} T_i$. In particular, 
            $$\partial V = T - \sum_{i=1}^N T_i= T- \sum_{i=1}^N \gamma_{i\#}\bb{\mathbb{S}^1}.$$
            This completes the proof. 
        \end{proof}

        Clearly, a finite sum of injective Lipschitz loops induces an element of $C_1(X)$. By abuse of notation, we denote the induced singular $1$-chain also by $T$, whenever $T$ is an integral $1$-cycle given by a finite sum of injective Lipschitz loops. We want to define a homomorphism $H_1^\textup{IC}(X) \to H_1(X)$ by sending each class $[\eta] \in H_1^\textup{IC}(X)$ to a class in $H_1(X)$ induced by a representative of $[\eta]$ given by finitely many Lipschitz loops. But to do so, we first need to show that the induced class in $H_1(X)$ does not depend on the representation of $[\eta]$. Fortunately, this can easily be done by factoring the different representations through the Riemannian manifold $M$. 

        \begin{lemma}\label{lemma: induced-sing-homology-class-doesnt-depend-on-represent}
            Let $\beta_i,\gamma_j\colon \mathbb{S}^1\to X$ be finitely many Lipschitz loops such that the integral $1$-cycles
            $$T = \sum_i^N \beta_{i\#}\bb{\mathbb{S}^1} \quad \textup{and} \quad S= \sum_j^M \gamma_{j\#}\bb{\mathbb{S}^1}$$
            define the same homology class in $H_1^\textup{IC}(X)$. Then, the induced singular homology classes $[T],[S]\in H_1(X)$ are equal.
        \end{lemma}

        %Recall that the homology groups $H_1^\textup{IC}(M)$ and $H_1(X)$ are isomorphic; see Theorem \ref{thme: sing-integral-homology-coincide}. In fact, an even stronger statement is true. A singular chain $c$ is called a singular Lipschitz chain if each singular simplex of $c$ is a Lipschitz map. The boundary of a singular Lipschitz chain is again a Lipschitz chain. Thus, we can define the singular Lipschitz homology groups $H_k^\textup{Lip}(X)$ via these Lipschitz chains. Notice that every singular Lipschitz chain induces an integral current. It follows that the inclusions from the space of singular Lipschitz chains into the space of integral currents and singular chains, respectively, are chain maps. In the smooth manifold $M$, these inclusion maps induce isomorphisms on the level of homology $H_k^\textup{Lip}(M) \to H_k(M)$ and $H_k^\textup{Lip}(M) \to H_K^\textup{IC}(M) $; see \cite[Theorem 1.3]{Mitsuishi}.

        In the Riemannian manifold $M$, the inclusions $C_1^\textup{Lip}(M)\hookrightarrow \bI_1(M)$ and $C_1^\textup{Lip}(M)\hookrightarrow C_1(M)$, respectively, induce isomorphisms on the level of homology; see Theorem \ref{thme: sing-integral-homology-coincide}.

        \begin{proof}
            It follows from Lemma \ref{lemma: lip-approx-easy} that there exists a Lipschitz map $\varphi\colon X \to M$ such that the composition $\psi = \varrho^{-1}\circ \varphi$ is homotopic to the identity on $X$. Therefore, $[\psi_\#T] = [T]$ and $[\psi_\#S]=[S]$ as singular homology classes. Since $S$ and $T$ are homologous as integral currents, the integral $1$-cycles $\psi_\# S$ and $\psi_\# T $ represent the same homology class in $H_1^\textup{IC}(M)$. Notice that $\varphi_\# T$ and $\psi_\#S$ are Lipschitz chains. Hence, the singular homology classes of $\varphi_\# T$ and $\psi_\#S$, obtained by the isomorphism $H_1^\textup{IC}(M) \to H_1(M)$ and the classes induced by interpreting $\varphi_\# T$ and $\psi_\#S$ as singular chains, coincide. Therefore, as singular homology classes, 
            $$[T] = [\psi_\#T] = \varrho^{-1}_*[\varphi_\#T] =  \varrho^{-1}_*[\varphi_\#S]= [\psi_\#S]= [S]$$
            This completes the proof.
        \end{proof}

        We define a homomorphism 
        \begin{equation}\label{eq: definition-homomorph-1-dim-homology}
            \Theta\colon H_1^\textup{IC}(X) \to H_1(X), \; [\eta] \mapsto \left[\sum_{i=1}^N \gamma_{i\#}\bb{\mathbb{S}^1} \right],
        \end{equation}
        where $\sum_{i=1}^N \gamma_{i\#}\bb{\mathbb{S}^1}$ is any representation of $[\eta]$ by finitely many Lipschitz loops. By Lemma \ref{lemma: integral-homology-repres-finite-loops-}, there exists such a representation for each class $[\eta] \in H_1^\textup{IC}(X)$. Furthermore, Lemma \ref{lemma: induced-sing-homology-class-doesnt-depend-on-represent} implies that $\Theta$ is well-defined. In order to prove Theorem \ref{thme: one-homology}, we have to show that $\Theta$ is bijective.

        \medskip

        So far, we have used representations of homology classes in $H_1^\textup{IC}(X)$ by Lipschitz loops. However, one can simply pass from such a representation to a representation by Lipschitz curves $\gamma\colon [0,1]\to X$ and vice versa.  In the remainder of this section, we will use whichever representation is more suitable for the situation.

        \begin{lemma}\label{lemma: codimension-one-homology-surface-surjective}
            The homomorphism $\Theta$ is surjective.
        \end{lemma}

        \begin{proof}
            Let $\varepsilon>0$ be such that every ball in $X$ with radius less than $\varepsilon$ is contractible in $X$. Such an $\varepsilon$ exists since $X$ is homeomorphic to a closed, smooth manifold. Let $[\eta] \in H_1(X)$ be a singular homology class, and let $c$ be a singular $1$-chain that represents $[\eta]$. Then $c$ is of the form $c = \sum_i^N n_i \cdot \gamma_i$, where $n_i \in \Z$ and $\gamma_i\colon [0,1]\to X$ are continuous curves. Fix $i = 1,\dots,N$ for the moment. Let $0=t_1<\dots<t_k= 1$ be a partition of $[0,1]$ such that each $\gamma_i([t_j,t_{j+1}])$ is contained within an open ball of radius $\varepsilon/2$. We define a Lipschitz curve $\beta_i \colon [0,1] \to X$ that is a geodesic from $\gamma_i(t_j)$ to $\gamma_i(t_{j+1})$ for every $j = 1,\dots, k-1$. Each concatenation $-\beta_i|_{[t_j,t_{j+1}]} \cdot \gamma_i|_{[t_j,t_{j+1}]}$ is a closed curve and is contained within a ball of radius $\varepsilon$, and hence is contractible. Here, $-\beta_i|_{[t_j,t_{j+1}]} \cdot \gamma_i|_{[t_j,t_{j+1}]}$ denotes the curve that first follows $\gamma_i|_{[t_j,t_{j+1}]}$ and then $\beta_i|_{[t_j,t_{j+1}]}$ in the reverse direction. We conclude that there exists a singular $2$-chain $C$ with $bC = \gamma_i -\beta_i$. We repeat this for each $i = 1,\dots,N$ and define $T = \sum_{i=1}^N n_i \cdot \beta_{i\#}\bb{0,1}\in \bI_1(X)$. Clearly, $\partial T = 0$ and $c$ and $T$ are homologous as singular chains. Therefore, $\Theta\big([T]\big) = [c] =[\eta]$. We conclude that $\Theta$ is surjective.
        \end{proof}

        \begin{lemma}\label{lemma: codimension-one-homology-surface-injective}
            The homomorphism $\Theta$ is injective.
        \end{lemma}

        \begin{proof}
            Let $[\eta]\in H_1^\textup{IC}(X)$ be such that $\Theta\big([\eta]\big)= 0 \in H_1(X)$. Therefore, there exists a representation of $[\eta]$ by finitely many Lipschitz curves $\gamma_i\colon [0,1]\to X$ and a singular $2$-chain $c$ that bounds $T =\sum_{i=1}^N \gamma_{i\#}\bb{0,1}$, when $T$ is interpreted as a singular $1$-chain. We may suppose that $\textup{length}(\gamma_i)<\Lambda/3$ for each $i =1,\dots, N$. Write  $c = \sum_{j=1}^m \theta_j \psi_j$, where $\psi_j \colon \Delta^2 \to X$ are continuous maps and $\theta_j$ is equal to $1$ or $-1$. By passing to barycentric subdivision and subdividing the $\gamma_i$ accordingly, if necessary, we may assume that for all $x,y \in \Delta^2$ and every $j =1,\dots,m$, we have $d(\psi_j(x),\psi_j(y))< \Lambda/3$. For $j =1,\dots,m$ and $k=1,2,3$, let $\psi_j^k\colon [0,1]\to X$ be the restriction of $\psi_j$ to the $k$th boundary component of $\Delta^2$. Then 
            \begin{equation}\label{eq: codimension-one-homology-surface-injective}
                \partial c = \sum_{j=1}^m \sum_{k=1}^3 (-1)^k  \theta_j \psi_j^k = \sum_{i=1}^N \gamma_i.
            \end{equation}
            For each $j =1,\dots,m$ and every $k=1,2,3$, we define a Lipschitz curve $\beta_j^k\colon [0,1]\to X$ as follows. If $\psi_j^k$ is equal to one of the curves $\gamma_i$, we set $\beta_j^k = \gamma_i$. Otherwise, we let $\beta_j^k\colon[0,1]\to X $ be a geodesic from $\psi^k_j(0)$ to $\psi^k_j(1)$. We can do this in a consistent way such that $\beta_j^k = \beta_{j'}^{k'}$ whenever $\psi^k_j = \psi_{j'}^{k'}$. For  every $j = 1,\dots,m$, let $S_j = \sum_{k=1}^3 (-1)^k \beta_{j\#}^k\bb{0,1}\in \bI_1(X)$. It follows
            $$\bM(S_j) \leq \sum_{k=1}^3 \textup{length}(\beta_j^k) < \Lambda,$$
            for all $j = 1,\dots,m$. Moreover, by construction, each $S_j$ is a cycle. Therefore, the small mass isoperimetric inequality implies that for every $j = 1,\dots,m$ there exists a filling $V_j\in \bI_2(X)$ of $S_j$. We put $V = \sum_{j=1}^m \theta_j V_j$.  By \eqref{eq: codimension-one-homology-surface-injective}, we have
            $$\partial V = \sum_{j=1}^m\theta_j \partial V_j = \sum_{j=1}^m \theta_j S_j = \sum_{j=1}^m\sum_{k=1}^3 (-1)^k \theta_j \beta_{k\#}^j\bb{0,1} = \sum_{i=1}^N \gamma_{i\#}\bb{0,1} = T.$$
            Hence, $V$ is a filling of $T$ and in particular, $[T] = [\eta]$ is the trivial class in $H_1^\textup{IC}(X)$. This completes the proof.
        \end{proof}
        
        By combining the two previous lemmas, we conclude that the homomorphism $\Theta \colon H_1^\textup{IC}(X)\to H_1(X)$ defined in \eqref{eq: definition-homomorph-1-dim-homology} is an isomorphism. This proves Theorem \ref{thme: one-homology}. We conclude the section with the proof of the corollary.
    
       \begin{proof}[Proof of Corollary \ref{cor: one-dim-homology-surf}]
            Let $X$ be a metric space homeomorphic to a closed, oriented, connected smooth surface. Suppose that $X$ is Ahlfors $2$-regular and linearly locally contractible. It follows from Proposition \ref{prop: sobolev-isop-implies-current-isoper} and Theorem \ref{thme: small-mass-isop-ahlfors-llc} that $X$ satisfies a $1$-dimensional small mass isoperimetric inequality for integral currents. Therefore, by Theorem \ref{thme: one-homology} the homology group $H_1^\textup{IC}(X)$ is isomorphic to $H_1(X)$. Furthermore, by \cite[Theorem 1.3]{basso2023geometric} there exists $T \in \bI_2(X)$ that generates $H_2^\textup{IC}(X)$ and in particular, $H_2^\textup{IC}(X)$ and $H_2(X)$ are isomorphic. Recall that $X$ is quasiconvex; see Section \ref{see: isoper}. Using this is it not difficult to show that $H_0^\textup{IC}(X)$ and $H_0(X)$ are isomorphic as well. Finally, $X$ is Ahlfors $2$-regular and thus $\bI_k(X)=0$ for all $k>2$. We conclude that $H_k^\textup{IC}(X) = 0 = H_k(X)$ for all $k >2$. This completes the proof. 
        \end{proof}
    
\section{Examples}\label{sec: examples}
    We construct two geodesic metric spaces $X$ and $Y$ that are homeomorphic to the $2$-sphere $\mathbb{S}^2$ and have finite Hausdorff $2$-measure. While $X$ is Ahlfors $2$-regular but not linearly locally contractible, $Y$ is doubling and linearly locally contractible but not Ahlfors $2$-regular. However, the $1$-dimensional homology groups via integral currents of $X$ and $Y$ are not trivial. 
    Recall from the introduction that a metric space $E$ is linearly locally contractible if there exists $\lambda>0$ such that every ball $B(x,r)$ in $E$ with $r\in (0,\diam( E)/ \lambda)$ is contractible within $B(x,\lambda r)$. Furthermore, $E$ is said to be Ahlfors $2$-regular if there exists $c>0$ such that for every ball $B(x,r)$ in $E$ with $r\in (0,\diam (E))$ we have
    $$c^{-1} r^2 \leq \Ha^2(B(x,r)) \leq cr^2.$$
    
    %In particular, the homomorphism in Theorem \ref{thme: codimension-one-homology} cannot be an isomorphism, and we do need the strong assumptions in Corollary \ref{cor: one-dim-homology-surf}.
    By \cite[Theorem 1.3]{basso2023geometric}, both spaces $X$ and $Y$ have a metric fundamental class. Therefore, these examples show that the homology groups via integral currents can be larger than the singular homology groups. 

    \subsection{The first example}
        Let $(S_k)_{k\geq 1}$ be a family of disjoint circles in $\mathbb{S}^2$ of radius $1/2^{3k}$, arranged in a "line" such that $d(S_k,S_{k+1}) = 1/2^k$ for all $k \geq 1$. We denote by $D_k$ the closed disk bounded by $S_k$. For $k \geq 1$, we define the "mushroom" $M_k$ with the stem equal to the cylinder $S_k\times [0,1/2^{3k}]$ and the cap equal to the hemisphere of radius $2^{-k}$.

        \begin{figure}[ht]
              \centering
              \includegraphics[width=0.75\textwidth]{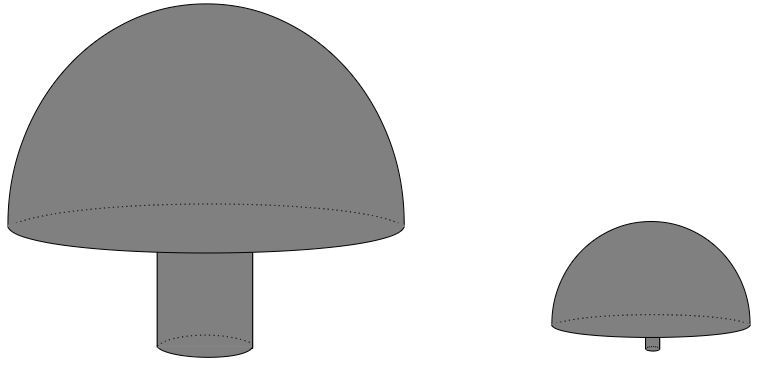}
              \caption{Two subsequent mushrooms $M_k$ and $M_{k+1}$}
        \end{figure}

        We obtain $X$ by replacing each disk $D_k$ with the mushroom $M_k$ and we equip $X$ with the intrinsic metric $d_{\textup{int}}$. Then
        $$\Ha^2(X) \leq \Ha^2(\mathbb{S}^2) + \sum_{k=1}^\infty \Ha^2(M_k) = 4\pi + 2\pi\sum_{k=1}^\infty  \frac{1}{2^{4k}} +\frac{1}{2^{2k}} < 5\pi.$$
        Therefore, $X$ is a geodesic metric space with finite Hausdorff $2$-measure that is homeomorphic to $\mathbb{S}^2$. We claim that $X$ is Ahlfors $2$-regular. The lower bound is easily verified. For the moment, fix $x \in X$ and $r\in (0,1)$. Let $j, N \geq1$ be such that $M_j$ is the closest mushroom to $x$ and $\sum_{k = N-1}^\infty 1/2^{k} \leq r \leq \sum_{k = N}^\infty 1/2^{k}$. In case $N \geq j+2$, then $B(x,r)$ does not intersect any other mushroom than $M_j$ and the upper bound $\Ha^2(B(x,r)) \preceq r^2$ is immediate. Otherwise, if $N < j+2$, then $B(x,r)$ intersects at most the $M_k$ for $k \geq N-1$, and hence,
        \begin{align*}
            \Ha^2(B(x,r)) &\leq cr^2 + \sum_{k=N-1}^\infty \Ha^2(M_k) \leq cr^2  + \sum_{k=N-1}^\infty \left(\frac{1}{2^k}\right)^2
            \\
            &\leq cr^2+ 4\pi\left(\sum_{k=N-1}^\infty \frac{1}{2^k}\right)^2 \leq (c + 4\pi) r^2, 
        \end{align*}
        %$$\Ha^2(B(x,r)) \leq Cr^2 + \sum_{k=N-1}^\infty \Ha^2(M_k) \leq Cr^2+ 4\pi\left(\sum_{k=N-1}^\infty \frac{1}{2^k}\right)^2 \leq (C + 4\pi) r^2,$$
        %\leq C r^2 + 4\pi \sum_{i=N-1}^\infty 2^{-2i} 
        where $c>0$ is the Ahlfors $2$-regularity constant of the standard $2$-sphere. A ball centered at some $S_k$ that wraps around the stem of $M_k$ is only contractible within a ball that contains the whole mushroom $M_k$. The ratio between the circumference of the stem and the height of $M_k$ is proportional to $1+2^{2k}$. We conclude that $X$ is not linearly locally contractible. Finally, for $k \geq 1$, we define $T = \sum_{k=1}^\infty 2^{2k}\bb{S_i} \in \bI_1(X)$. Then, $T$ is a cycle and
        $$\bM(T) \leq \sum_{k=1}^\infty  2^{2k}\bM(\bb{S_k})= 2\pi \sum_{k=1}^\infty \frac{1}{2^{k}} < \infty.$$
        Notice that each $\bb{S_k}$ has only two fillings, $\bb{M_k}$ and $\bb{M^c_k}$. Out of these two possibilities, $\bb{M_k}$ has smaller mass.
        However, 
        $$\sum_{k=1}^\infty 2^{2k} \bM(\bb{M_k}) = 2\pi \sum_{k=1}^\infty 2^{2k} \left(\frac{1}{2^{4k}} +\frac{1}{2^{2k}}\right) = \infty. $$
        It follows that $T$ has no filling with finite mass. In particular, $[T]\neq 0 $ and $H_1^\textup{IC}(X)$ is not trivial.

    \subsection{The second example}
        We construct the next example $Y$ out of the surface $Q = \partial[0,1]^3$ of the unit cube in $\R^3$. As before, we replace small pieces in $Q$ with more complicated sets to obtain $Y$. We begin with the definition of those sets. Divide the unit square $[0,1]^2$ into $9$ squares of side-length $1/3$ and replace the middle square with the surface of the cube $[0,1/3]^3$, from which we have removed the bottom square. The resulting surface $G$ is called the generator and consists of $13$ squares of side-length $1/3$. Set $S_0 = G$. We inductively define the surfaces $S_k$ by replacing each of the $13^k$ squares in $S_{k-1}$ with a copy of the generator scaled by a factor of $1/3^k$.

         \begin{figure}[ht]
              \centering
              \includegraphics[width=1\textwidth]{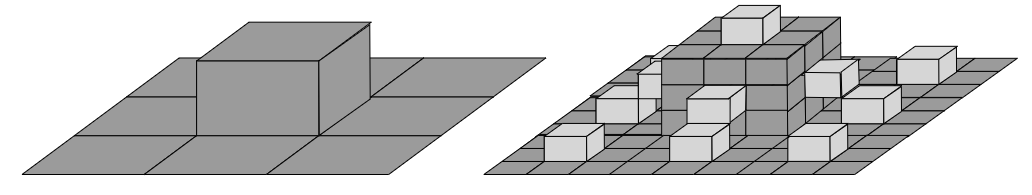}
              \caption{The generator $G$ and the first iteration $S_1$.}
        \end{figure}

         Notice that the $S_k$ have self-intersections. To avoid this, we carefully round off the sides of each new cube added in $S_k$. We equip each $S_k$ with the intrinsic length metric $d_{\textup{int}}$. It is not difficult to show that the sequence $(S_k)_{k\geq 0}$ converges in the Hausdorff distance. The limit $\mathcal{S}$ is one side of the so-called snowsphere. The snowsphere is a $2$-dimensional analogue of the snowflake curve and has been extensively studied in \cite{meyer-snowsphere-1, meyer-snowsphere-2}. We need the following facts. 
                \begin{enumerate}
                    \item The $S_k$ are uniformly doubling; 
                    %(\cite[proof of Theorem 4.2.]{meyer-snowsphere-1});
                    \item $\Ha^2(S_k) = \left(13/9\right)^k$ for every $k \in \N$ and in particular, $\Ha^2(S_k) \to \infty$ as $k \to \infty$;
                    \item $d_{\textup{int}}(x,y) \leq 4$ for all $x,y\in S_k$ and $k\in \N$;
                    \item the $S_k$ are uniformly linearly locally contractible.
                \end{enumerate}
                The first three properties are verified easily and we refer to \cite{meyer-snowsphere-1} for details. We explain why the last property holds. By \cite[Theorem 5.1]{meyer-snowsphere-1}, the limit $\mathcal{S}$ of the $S_k$ is homeomorphic to the unit square via a quasisymmetry $f\colon \mathcal{S}\to [0,1]^2$. That is, there exists a homeomorphism $\eta\colon [0,\infty) \to [0,\infty)$ such that 
                \begin{equation}\label{eq: def-quasisymmetry}
                    d_\mathcal{S}(x,y)\leq t d_\mathcal{S}(x,z)\quad \implies\quad \norm{f(x)-f(y)} \leq \eta(t) \norm{f(x)-f(z)}
                \end{equation}
                for all $x,y,z \in \mathcal{S}$ and $t\geq 0$. It follows from the construction of $f$ that the $S_k$ are uniformly quasisymmetric to the unit square. Here, uniformly means that there exists one homeomorphism $\mu\colon [0,\infty) \to [0,\infty)$ such that each quasisymmetry $f_k \colon S_k \to [0,1]^2$ satisfies \eqref{eq: def-quasisymmetry} with respect to that homeomorphism $\mu$. Furthermore, it is a direct consequence of the definitions that a compact metric space $X$ which is quasisymmetric to a linearly locally contractible metric space is itself linearly locally contractible (quantitatively). This yields property (4). 
                %Alternatively, one can use the following fact to prove (d). The four corners of the generator $G$ can be connected to a corner of the surface attached in the construction of $G$ by a curve with length at most $\frac{4}{3}$.
                Let $Q$ be the surface of the unit cube $[0,1]^3$ and let $(Q_k)_{k\in \N}$ be a family of disjoint squares of side-length $\left(3/4\right)^k$. We obtain $Y$ by replacing each square $Q_k$ with $S_k$ (scaled appropriately). Then, $Y$ is a topological $2$-sphere and has finite Hausdorff $2$-measure 
                $$\Ha^2(Y) \leq \Ha^2(Q) + \sum_{k=1}^\infty \left(\frac{3}{4}\right)^{2k} \Ha^2(S_k) = 6 +\sum_{k=1}^\infty \left(\frac{13}{16}\right)^k <\infty.$$ 
                Combining properties (2) and (3) implies that $Y$ is not Ahlfors $2$-regular. Since the doubling property and linear local contractibility are invariant under scaling, we conclude that $Y$ is doubling and linearly locally contractible. Finally, let $T = \sum_{k=1}^\infty \left(\frac{13}{10}\right)^k \partial\bb{S_k'}$, where the $S_k'$ denote the scaled copies of the $S_k$ we attached to $Y$. Then, $T$ is an integral $1$-cycle with
                $$\bM(T) \leq  \sum_{k=1}^\infty \left(\frac{13}{10}\right)^k \bM(\partial\bb{S_k'}) = 4 \sum_{k=1}^\infty \left(\frac{39}{40}\right)^k <\infty. $$
                For each $k \geq 1$, the integral $1$-cycle $\partial\bb{S_k'}$ has exactly two fillings, $\bb{S_k'}$ and $\bb{(S_k')^c}$. We have $\bM(\bb{S_k'} )\leq \bM(\bb{(S_k')^c})$ for each $k \geq 1$.
                However,
                $$\sum_{k=1}^\infty \left(\frac{13}{10}\right)^k \bM(\bb{S_k'}) = \sum_{k=1}^\infty \left(\frac{169}{160}\right)^k= \infty.$$
                Therefore, $T$ has no filling with finite mass. We conclude that $[T]\neq 0 $ and $H_1^\textup{IC}(X)$ is not trivial.

\bibliographystyle{plain}

\end{document}